\documentclass[11pt]{article}

\usepackage{amsmath,amsthm,amssymb,amscd}

\usepackage{float}
\usepackage[caption = false]{subfig}
\usepackage[final]{graphicx}

\usepackage{xcolor}
\usepackage{todonotes}
\usepackage[colorlinks,citecolor = red,urlcolor=blue]{hyperref}

\def \RR {{\mathbb{R}}}
\def \EE {{\mathbb{E}}}
\def \VV {{\mathbb{V}}}
\def \PP {{\mathbb{P}}}

\def \D {{\rm d}}

\def \one {{\bf 1}}

\def \tP {{\widetilde{P}}}
\def \tQ {{\widetilde{Q}}}
\def \tZ {{\widetilde{Z}}}

\def \tC {{\widetilde{C}}}
\def \tV {{\widetilde{V}}}

\def \tX {{\widetilde{X}}}
\def \tXn {{\widetilde{X}_n}}
\def \tXm {{\widetilde{X}_m}}
\def \tXnm {{\widetilde{X}_{n-1}}}
\def \tXnp {{\widetilde{X}_{n+1}}}
\def \tXfnp {{\widetilde{X}^{\rm f}_{n+1}}}
\def \tXcnp {{\widetilde{X}^{\rm c}_{n+1}}}
\def \tXfn {{\widetilde{X}^{\rm f}_n}}
\def \tXf  {{\widetilde{X}^{\rm f}}}
\def \tXcn {{\widetilde{X}^{\rm c}_n}}
\def \tXc  {{\widetilde{X}^{\rm c}}}
\def \tXcl {{\widetilde{X}^{\rm c}_{\underline{n}}}}

\def \hP {{\widehat{P}}}

\def \hY {{\widehat{Y}}}

\def \hX {{\widehat{X}}}
\def \hXn {{\widehat{X}_n}}

\def \hXnp {{\widehat{X}_{n+1}}}
\def \hXfnp {{\widehat{X}^{\rm f}_{n+1}}}
\def \hXcnp {{\widehat{X}^{\rm c}_{n+1}}}
\def \hXcpp {{\widehat{X}^{\rm c}_{n+2}}}
\def \hXfn {{\widehat{X}^{\rm f}_n}}
\def \hXf {{\widehat{X}^{\rm f}}}
\def \hXcn {{\widehat{X}^{\rm c}_n}}
\def \hXc {{\widehat{X}^{\rm c}}}
\def \hXcl {{\widehat{X}^{\rm c}_{\underline{n}}}}

\def\eps{{\varepsilon}}

\newcommand{\fracs}[2]{{\textstyle \frac{#1}{#2}}}

\newtheorem{theorem}{Theorem}[section]
\newtheorem{lemma}[theorem]{Lemma}

\newtheorem{corollary}[theorem]{Corollary}
\newtheorem{assumption}[theorem]{Assumption}

\begin{document}

\title{Analysis of nested multilevel Monte Carlo \\ using approximate Normal random variables}

\author{
Mike Giles
\thanks{\href{mailto:mike.giles@maths.ox.ac.uk}%
{\texttt{mike.giles@maths.ox.ac.uk}}}
\and 
Oliver Sheridan-Methven
\thanks{\href{mailto:oliver.sheridan-methven@hotmail.co.uk}%
{\texttt{oliver.sheridan-methven@hotmail.co.uk}}}
}

\maketitle

\begin{abstract}
  The multilevel Monte Carlo (MLMC) method has been used for a wide variety
  of stochastic applications.  In this paper we consider its use in situations
  in which input random variables can be replaced by similar approximate random
  variables which can be computed much more cheaply.  A nested MLMC approach is
  adopted in which a two-level treatment of the approximated random variables
  is embedded within a standard MLMC application. We analyse the resulting
  nested MLMC variance in the specific context of an SDE discretisation in
  which Normal random variables can be replaced by approximately Normal
  random variables, and provide numerical results to support the analysis.
\end{abstract}

\section{Introduction}

Following the initial work of Heinrich \cite{heinrich98} on parametric
integration and Giles \cite{giles08} on stochastic differential equations
(SDEs), there has been huge development in the application of multilevel
Monte Carlo (MLMC) methods to a wide variety of stochastic modelling
contexts. This includes partial differential equations (PDEs) with
stochastic coefficients \cite{cgst11,bsz11}, stochastic PDEs \cite{gr12},
continuous-time Markov process models of biochemical reactions
\cite{ah12,ahs14},
Markov Chain Monte Carlo \cite{hss13,sst17},
nested simulation \cite{bhr15,gg19},
probability density estimation \cite{gnr15,bc16},
and reliability estimation \cite{up15,ehm16}.
A review of research on MLMC is provided by Giles \cite{giles15}.

In this paper we are concerned with the development and analysis of a new
class of MLMC methods involving approximate probability distributions.
Most numerical methods for simulating stochastic models start from random
inputs from a variety of well-known distributions: Normal, Poisson, binomial,
non-central $\chi^2$, etc.  Generating samples which have a distribution
which matches the desired distribution to within the limits of finite precision
arithmetic can be a significant part of the overall computational cost of
the simulation.
Here we consider what can be achieved if it is also possible to generate
approximate random variables (random variables with a distribution
which is only approximately correct) at a greatly reduced cost.
We will show that a nested MLMC approach can be adopted in which a two-level
treatment of the approximated random variables is embedded within a standard
MLMC application. We then analyse the MLMC variance of the resulting treatment
in the specific context of an SDE discretisation in which Normal random variables
can be replaced by approximately Normal random variables.

The most relevant prior research is the work of Giles, Hefter, Mayer and Ritter
\cite{ghmr19,ghmr19b}.  
This research was in the context of \emph{Information Based Complexity} (IBC), 
working with a complexity model which counted the number of
individual random bits, rather than viewing each standard
uniformly-distributed random
variable as having a unit cost.  Fundamental to the algorithms in these papers
was the use of quantised Normal random variables, which is the first of the
three approximations to be discussed in the next section.  Some elements of the
analysis in this paper build on the ideas and analysis in those papers, but
the context is quite different in aiming to minimise the real-world execution
cost of MLMC algorithms on modern CPUs and GPUs, and the specifics of the
proposed method are quite different in using a nested MLMC approach.

Another relevant paper is by M\"uller et al.~\cite{mss15}.
In this work they use three- and four-point approximations to the Normal
distribution, equivalent to a piecewise constant approximation of $\Phi^{-1}(U)$
on 3 or 4 intervals of a non-uniform size, chosen so that the leading moments
are the same as for the standard Normal distribution.  The way in which these
are used within the MLMC construction violates, to a small extent, the usual
telescoping summation which lies at the heart of the MLMC method, and therefore
they have to be careful to bound the magnitude of this error. This new error is
related to the fact that if $Z_1$ and $Z_2$ are unit Normal random variables,
then so too is their sum $(Z_1{+}Z_2)/\sqrt{2}$; in an SDE application, this
is important in MLMC so that the sum of two Brownian increments from timesteps
of size $h$ corresponds to a Brownian increment from a timestep of size $2h$.
However, this is no longer true when using approximate Normal distributions;
the sum $(\tZ_1{+}\tZ_2)/\sqrt{2}$ is still an approximation of a unit Normal
random number, but in general it does not come from exactly the same distribution
as $\tZ_1$ and $\tZ_2$.

Similar ideas have also been investigated by Belomestny and Nagapetyan \cite{bn17}
who avoid errors in the telescoping summation by using different approximate
distributions on each level of MLMC refinement.  However, in the present paper
we prefer to avoid these difficulties entirely by using the same approximate
distribution throughout within a nested MLMC treatment; we think this will
generalise better to different approximations well-suited to different
computer hardware, and to applications with more complex distributions.
We also consider three different kinds of approximations, some of which
will generalise better to distributions such as the non-central
$\chi^2$-distribution in which there are additional parameters which may
vary for each random number sample; in such cases a lookup table based
on quantisation could become unreasonably large. 

\section{Approximate Normal distributions}

\label{sec:approximations}

There are many ways in which approximate Normal variables can be generated.
In this paper we consider three methods, each of which can be viewed as an
approximation of the generation of Normal random variables
through the inversion of the Normal CDF function,
$\displaystyle
Z = \Phi^{-1}(U),
$
with $U$ being a uniform random variable on the unit interval $(0,1)$.
The corresponding approximations all have the form
$\displaystyle
\tZ = \tQ(U),
$
so that it is possible to generate coupled pairs $(Z,\tZ)$ 
from the same random input $U$.

The three approximations are all motivated by the different hardware features 
of modern CPUs and GPUs (manycore graphics processing units). Their analysis, 
implementation details, and resulting execution performance are explained more fully in \cite{sheridanmethven2020approximating,sheridan-methven2020thesis}.

Note that we are not concerned with the computational cost of generating the 
uniform random numbers $U$.  This is because we can follow 
Giles et al.~\cite{ghmr19b} in using a trick due to 
Bakhvalov \cite{bakhvalov64} to very efficiently 
generate a set of uniformly-distributed values $\{ U_1, U_2, \ldots\}$ 
with pairwise independence (i.e.~for any two indices $i\neq j$, $U_i$ 
is independent of $U_j$) at a cost of less than 3 computer operations, 
on average.
In essence, the procedure is very similar to the \emph{digital shift} 
used to transform Sobol' points in a randomised Quasi Monte Carlo 
computation \cite{ecuyer2016randomized}.

\begin{figure}
\begin{center}
\subfloat[Quantised approximation]{\includegraphics[width = 4.5in]{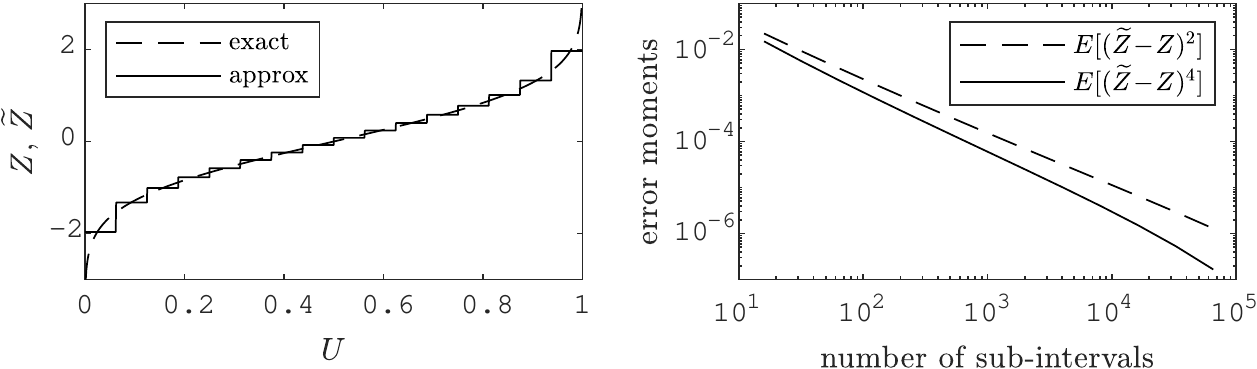}} \\
\subfloat[Piecewise linear approximation on dyadic intervals]{\includegraphics[width = 4.5in]{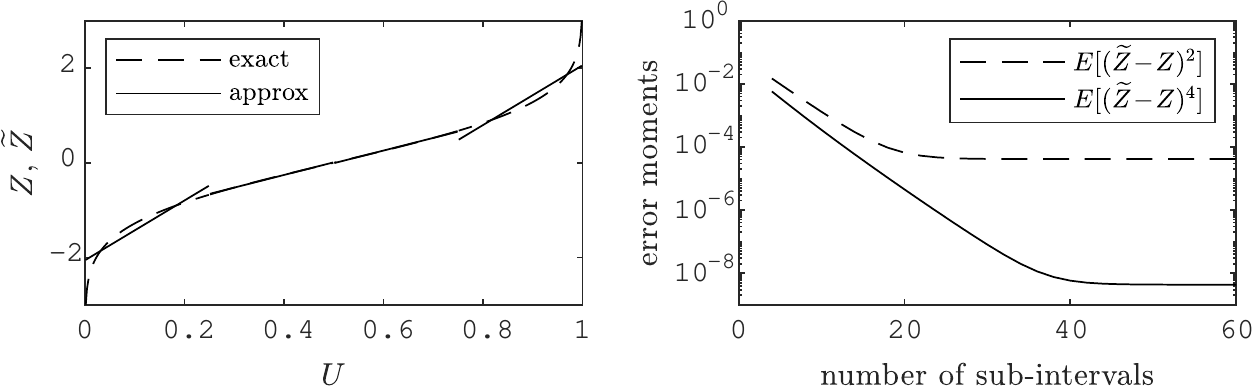}} \\
\subfloat[Polynomial approximation]{\includegraphics[width = 4.5in]{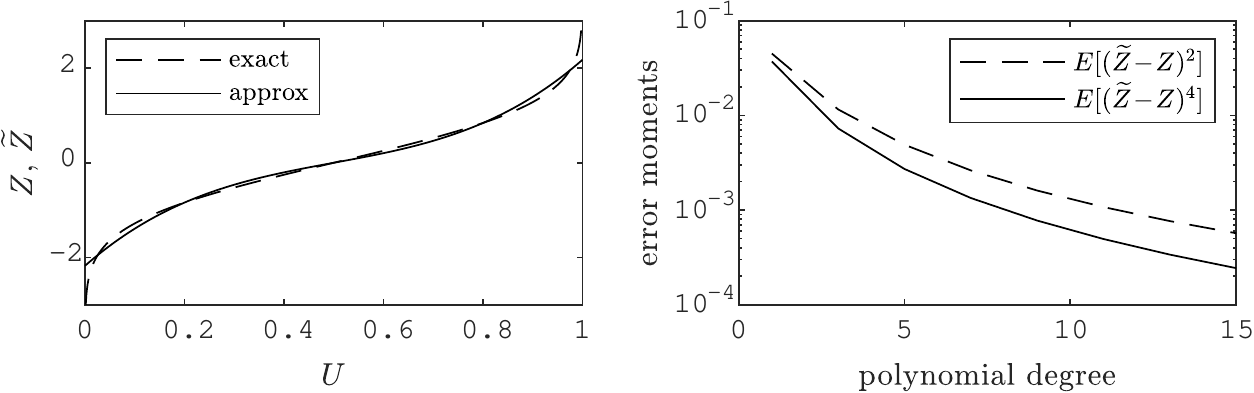}}
\end{center}
\caption{Three approximations of the inverse Normal CDF.}
\label{fig:fig1}
\end{figure}

\subsection{Quantised approximation}

The first approximation is a simple piecewise constant approximation on
$K\!=\!2^q$ intervals of size $2^{-q}$, with the value $\tQ(U)$ on the
$k$-th interval
$\displaystyle
I_k = [(k{-}1)\, 2^{-q}, k\, 2^{-q}],
$
given either by the average value of $\Phi^{-1}(U)$ on the interval, or 
alternatively the value at the mid-point.

The index $k$ corresponds to the first $q$ bits of the binary expansion 
for $U$, and this quantised approximation is the one considered by 
Giles et al.~\cite{ghmr19,ghmr19b}
since each random variable $\tZ$ can be generated based directly on 
$q$ random bits, each independently taking the value 0 or 1 with equal 
probability.  Extending their analysis, it can be proved that
\[
\EE[\, |\tZ{-}Z|^p] = o(2^{-q} ).
\]

Figure \ref{fig:fig1}(a) illustrates the approximation for $K\!=\!16$, 
and also has plots of $\EE[(\tZ{-}Z)^2]$ and $\EE[(\tZ{-}Z)^4]$ 
as a function of $K$, the number of intervals.

For $K\!=\!1024$ the mean square error (MSE) is
$\EE[(\tZ{-}Z)^2] \!\approx\! 1.5\!\times\! 10^{-4}$.
This size seems a good choice as the lookup table will fit inside the L1 
cache of a current generation Intel CPU, leading to a very efficient scalar 
implementation.

\subsection{Piecewise linear approximation on dyadic intervals}

The second approximation uses a discontinuous piecewise linear approximation
on a geometric sequence of sub-intervals.  To be specific, for a given ratio 
$\fracs{1}{2}\leq r < 1$, $K$ sub-intervals on $[0,\fracs{1}{2}]$ are defined by
\[
I_k = [\fracs{1}{2} r^{k},\fracs{1}{2} r^{k-1}], ~~ k = 1, 2, \ldots K{-}1,  ~~~~~~~
I_K = [0, \fracs{1}{2} r^{K-1}].
\]

In the particular case $r\!=\!\fracs{1}{2}$, given an input $0\!<\!U\!<\!\fracs{1}{2}$, 
the corresponding sub-interval index $k$ can be determined by computing the integer 
part of $\log_2U$, which can be implemented very efficiently due to the floating 
point format of real numbers.

On each sub-interval $I_k$, $\tQ(U)$ is defined as the least-squares linear best 
fit approximation to $\Phi^{-1}(U)$, and the approximation on $[\fracs{1}{2},1]$ 
is defined by $\tQ(U) = - \tQ(1{-}U)$.  Standard analysis of the accuracy of
piecewise linear interpolation leads to the result that
\[
\EE[\, |\tZ{-}Z|^p] = o(r^K) + O( (1{-}r)^{2p}),
\]
where the first term comes from the two end intervals $[0,\fracs{1}{2} r^{K-1}]$ 
and $[1{-}\fracs{1}{2} r^{K-1}, 1]$, and the second term comes from the other intervals.
Note that to achieve convergence to zero requires that both $r\rightarrow 1$
and $r^K\rightarrow 0$.

Figure \ref{fig:fig1}(b) illustrates the approximation for 
$r\!=\!\fracs{1}{2} , K\!=\!2$, 
and also has plots of $\EE[(\tZ{-}Z)^2]$ and $\EE[(\tZ{-}Z)^4]$ 
as a function of $2K$,  the number of intervals, with fixed
$r\!=\!\fracs{1}{2}$.  Note that for both $\EE[(\tZ{-}Z)^2]$
and $\EE[(\tZ{-}Z)^4]$ the $(1{-}r)^{2p}$ error term eventually
dominates once $r^K$ is sufficiently small.

For $K\!=\!16$ the MSE is
$\EE[(\tZ{-}Z)^2]\!\approx\! 4\!\times\! 10^{-5}$.
This size is significant because in single precision it is possible to 
perform the table lookup within a single 512-bit AVX vector register on 
current generation Intel Xeon CPUs, giving a very efficient vector 
implementation.

\subsection{Polynomial approximation}

The final method is a simple polynomial approximation
\[
\tQ(U) \equiv \sum_{k=1}^K a_k\, (U{-}\fracs{1}{2})^{2k-1},
\]
with the coefficients $a_k$ determined by a least squares best fit 
to $\Phi^{-1}(U)$.  This method is particularly efficient on GPUs 
as it avoids the need for a table lookup, however it is also the 
least accurate of the three approximations for realistic polynomial sizes.

Figure \ref{fig:fig1}(c) illustrates the approximation for $K\!=\!2$, 
and also has plots of $\EE[(\tZ{-}Z)^2]$ and $\EE[(\tZ{-}Z)^4]$ 
as a function of $2K{-}1$, the degree of the polynomial.

For $K\!=\!4$ the MSE is $\EE[(\tZ{-}Z)^2]\!\approx\! 2.6 \!\times\!10^{-3}$;
this seems a good balance between computational cost and accuracy,
and will be used in the numerical experiments later.

\section{MLMC algorithms}

In this section we begin with a quick recap of the multilevel Monte Carlo
method, and then discuss how a nested version of MLMC can be used with
approximate distributions. The third part then applies the ideas to the
simulation of SDE solutions, using approximate Normal random variables.

\subsection{Standard MLMC}

\label{sec:standard_MLMC}

If $P$ is a random variable which is a function of a set of random inputs
$\omega$, then the Monte Carlo estimate for the expected value $\EE[P]$ is
the simple average
\[
N^{-1} \sum_{n=1}^N P(\omega^{(n)})
\]
where the $\omega^{(n)}, n\!=\!1, 2, 3, \ldots, N$ are i.i.d.~samples of $\omega$.
To achieve a root-mean-square (RMS) error of $\eps$ requires
$\displaystyle
N \!\approx\! \eps^{-2} V
$
samples, where $V \!=\! \VV[P]$ is the variance.  If each sample costs $C$
then the total cost is approximately $\eps^{-2} V \, C$.

Suppose now that $\tP \approx P$, then since
$\displaystyle \EE[P] = \EE[\tP] + \EE[P{-} \tP]$
we can instead use the estimator
\[
N_0^{-1} \sum_{n=1}^{N_0} \tP(\omega^{(0,n)}) + 
N_1^{-1} \sum_{n=1}^{N_1} (P(\omega^{(1,n)}) - \tP(\omega^{(1,n)})).
\]
The cost of this estimator is $N_0 C_0 + N_1 C_1$, and 
the variance is $V_0 / N_0 + V_1 / N_1$, where
$V_0\equiv \VV[\tP], ~ V_1\equiv \VV[P{-}\tP]$, if all of 
the $\omega^{(\ell,n)}$ are independent.
Minimising the cost subject to the same accuracy requirement 
gives the total cost
$\displaystyle
\eps^{-2} (\sqrt{V_0C_0} + \sqrt{V_1C_1})^2.
$
which is significantly less than $\eps^{-2} V \, C$ if $C_0\ll C$ and $V_1\ll V$.

To give a quantitative example, suppose that
$C_0 \!=\! 10^{-1} C,\ C_1 \!=\! C$, so the cost of evaluating $\tP$ is
10 times less than evaluating $P$, and
$V_0 \!=\! V,\ V_1 \!=\! 10^{-3} V$.
In that case, the total cost is approximately $0.12\, \eps^{-2} V\, C$,
a factor of 8 savings compared to the original Monte Carlo calculation.

This two-level calculation is easily generalised to a multilevel treatment.
Given a sequence of increasingly accurate (and costly) approximations 
$\hP_0, \hP_1, \hP_2, \ldots \longrightarrow P$, for example from the
approximation of an SDE using $2^\ell$ timesteps on level $\ell$, then
\[
\EE[ \hP_L ] = \EE[ \hP_0 ] + \sum_{\ell=1}^L \EE[ \hP_\ell {-} \hP_{\ell-1} ],
\]
and so the MLMC estimate for $\EE[\hP_L]$ is
\[
\hY \equiv N_0^{-1}  \sum_{n=1}^{N_0} \hP_0^{(n)} + 
\sum_{\ell=1}^L N_\ell^{-1}  \sum_{n=1}^{N_\ell} (\hP_\ell^{(\ell,n)} - \hP_{\ell-1}^{(\ell,n)}).
\]
The expected value of the estimator is $\EE[\hP_L]$, and the MSE
can be decomposed into the sum of the variance and the square of the bias
to give
\[
  {\rm MSE}\ =\  
  \VV[\hY] + \left( \EE[\hP_L - P] \right)^2
  = \sum_{\ell=0}^L N^{-1}_\ell V_\ell + \left( \EE[\hP_L - P] \right)^2,
\]
where $V_0\equiv\VV[\hP_0]$, and for $\ell\!\geq\!1$, 
$V_\ell \equiv \VV[\hP_\ell {-}\hP_{\ell-1}]$.

If level $L$ is chosen so that $|\EE[\hP_L - P]| \!<\! \eps/\sqrt{2}$, then
an overall RMS error of $\eps$ can be achieved at a total cost of approximately
\[
  2\, \eps^{-2} \left( \sum_{\ell=0}^L \sqrt{V_\ell C_\ell} \right)^2,
\]
where $C_\ell$ is the cost of a single sample of $\hP_\ell {-}\hP_{\ell-1}$.  
If the product $V_\ell C_\ell$ decreases exponentially with level then the overall 
cost is $O(\eps^{-2})$, corresponding to an $O(1)$ cost per sample on average, 
much less than for the standard Monte Carlo method.
For further details see \cite{giles08,giles15}.

\subsection{Nested MLMC for approximate distributions}

\label{sec:nested_MLMC}

When using random variables from approximate distributions, the first possibility is
to use the two-level treatment discussed before, with
\[
\EE[ P ]\ =\ \EE[ \tP ]\ +\ \EE[ P {-} \tP ].
\]
In this case, each sample $P {-} \tP$ would use the same underlying stochastic
sample for both $P$ and $\tP$, for example using the same input uniform random
variable $U$, and then applying the inverse of either the true CDF or an approximate CDF to
produce the random variables required to compute $P$ and $\tP$, respectively.
As stated before, this can give considerable savings if the cost of computing $\tP$
is much less than the cost of computing $P$, and $\VV[P{-}\tP] \ll \VV[P]$.

However, what can we do if our starting point is an MLMC expansion in some 
other quantity, such as the timestep?  In that case we can use nested MLMC,
an idea discussed by Giles, Kuo and Sloan \cite{gks18} which is a generalisation of 
Multi-Index Monte Carlo by Haji-Ali, Nobile and Tempone \cite{hnt16}.
We start from the usual MLMC expansion and then split each of the required
expectations into two pieces, one using the approximate random variables and the
other computing the required correction,
\begin{eqnarray*}
\EE[ \hP_L ] &=& \EE[ \hP_0 ] + \sum_{\ell=1}^L \EE[ \hP_\ell {-} \hP_{\ell-1} ]\\
           &=& \EE[ \tP_0 ] + \EE[ \hP_0 {-}\tP_0 ] \\
& & +\ \sum_{\ell=1}^L \left\{\rule{0in}{0.22in}  \EE[ \tP_\ell {-} \tP_{\ell-1} ] + 
\EE\left[ (\hP_\ell {-} \hP_{\ell-1}) - (\tP_\ell {-} \tP_{\ell-1}) \right] \right\}
\end{eqnarray*}

The pair $(\tP_\ell, \tP_{\ell-1})$ are generated in the same way as
$(P_\ell, P_{\ell-1})$, based on the same underlying uniform random variables $U$,
but converting them differently into the random variables required for the
calculation of $\hP$ and $\tP$.

As explained in the previous section, the standard
MLMC cost to achieve an RMS accuracy of $\eps$ is approximately
\[
C_{\rm MLMC} = 2\, \eps^{-2} \left( \sum_{\ell=0}^L \sqrt{C_\ell V_\ell} \right)^2,
\]
where $C_\ell$ is the cost of a single sample of $\hP_\ell{-}\hP_{\ell-1}$
on level $\ell$, and $V_\ell$ is its variance.
The complexity analysis extends naturally to the nested
{\em approximate} MLMC (AMLMC)
algorithm described above, giving a cost of approximately
\[
C_{\rm AMLMC} = 2\, \eps^{-2} \left( \sum_{\ell=0}^L \sqrt{\tC_\ell V_\ell}
    + \sqrt{(C_\ell{+}\tC_\ell) \tV_\ell} \right)^2,
\]
where $\tC_\ell$ is the cost of one sample of $\tP_\ell{-}\tP_{\ell-1}$,
which has a variance approximately equal to $V_\ell$, and
$C_\ell{+}\tC_\ell$ is the cost of one sample of 
$(\hP_\ell{-}\hP_{\ell-1}) - (\tP_\ell{-}\tP_{\ell-1})$, and $\tV_\ell$ is
its variance.

Note that 
\begin{eqnarray*}
  C_{\rm AMLMC} &=& 2\, \eps^{-2} \left( \sum_{\ell=0}^L \sqrt{C_\ell V_\ell}
  \sqrt{\frac{\tC_\ell}{C_\ell}} \left( 1 
  + \sqrt{\left(\frac{C_\ell}{\tC_\ell}{+}1\right) \frac{\tV_\ell}{V_\ell}}\ \right)
   \right)^2\\
  &\leq & C_{\rm MLMC} \max_{0\leq\ell\leq L} \frac{\tC_\ell}{C_\ell}
  \left( 1+ \sqrt{\left(\frac{C_\ell}{\tC_\ell}{+}1\right) \frac{\tV_\ell}{V_\ell}}\ \right)^2
\end{eqnarray*}
so that if $\displaystyle \tV_\ell/V_\ell \ll \tC_\ell/C_\ell \ll 1$
then the cost is reduced by a factor of approximately
$\max_\ell  \tC_\ell/C_\ell$.

\subsection{Application to SDE simulation}

To make the ideas in the preceding section more concrete, we consider
an application involving the solution of a scalar autonomous SDE,
\[
\D X_t = a(X_t)\, \D t + b(X_t)\, \D W_t,
\]
on the time interval $[0,T]$ subject to fixed initial data $X_0$.
Furthermore, we suppose that we are interested in the expected
value of a function of the final path value $\EE[f(X_T)]$.

If $\hXn$ is an approximation to $X_{nh}$ using a uniform timestep of size $h$,
then the simplest numerical approximation is the Euler-Maruyama discretisation,
\[
\hXnp = \hXn + a(\hXn)\, h + b(\hXn)\, \Delta W_n,
\]
in which the Brownian increment $\Delta W_n$ is a Normal random
variable with mean 0 and variance $h$, so it can be simulated as
$
\Delta W_n \equiv \sqrt{h}\, Z_n
$
where $Z_n$ is a unit Normal random variable.  In turn $Z_n$ can 
be generated from a uniform $(0,1)$ random variable $U_n$ through 
$Z_n = \Phi^{-1}(U_n)$.
Using this discretisation the output quantity of interest would
be $\hP \equiv f(\hX_N)$ where $N = T/h$ is assumed to be an integer.

The simplest way in which we could use approximate Normal random
variables would be to keep to a fixed number of timesteps, and
generate an approximate output quantity $\tP$ by replacing the
Normals $Z_n$ by approximate Normals $\tZ_n \equiv \tQ(U_n)$
generated using the same $U_n$.

However, we would like to combine the benefits of MLMC for the time
discretisation with the reduced execution cost of approximate
Normals and so can instead use the nested MLMC approach.
To do this, the key question is how do we compute each sample of
$(\hP_\ell {-} \hP_{\ell-1}) - (\tP_\ell {-} \tP_{\ell-1})$?

Let $h\equiv h_\ell$ be the timestep for a fine path on level $\ell$,
with $\hXfn$ representing the fine path approximation to the SDE
solution $X(nh)$ which is computed using the discrete equations
\begin{equation}
\hXfnp = \hXfn + a(\hXfn)\, h + b(\hXfn)\, \sqrt{h}\, Z_n,
\label{eq:hf}
\end{equation}
based on the true Normals, $Z_n$.  The corresponding coarse path 
approximation $\hXcn$ using timesteps of $2h$ and combined
Brownian increments $\Delta W_n {+} \Delta W_{n+1}$ is given by 
\begin{eqnarray*}
\hXcpp
   &=& \hXcn + 2\, a(\hXcn)\, h + b(\hXcn)\, (\Delta W_n + \Delta W_{n+1})
\\ &=& \hXcn + 2\, a(\hXcn)\, h + b(\hXcn)\, (\sqrt{h}\, Z_n + \sqrt{h}\, Z_{n+1}),
\end{eqnarray*}
for even integers $n$.  Alternatively, it is more convenient 
%for analysis 
to write it equivalently as
\begin{equation}
\hXcnp = \hXcn + a(\hXcl)\, h + b(\hXcl)\, \sqrt{h}\, Z_n,
\label{eq:hc}
\end{equation}
where $\underline{n} \equiv 2 \lfloor n/2 \rfloor$ is $n$ rounded 
down to the nearest even number. This gives the same values 
for $\hXcn$ at the even timesteps.

The corresponding discrete equations for the fine and coarse paths 
computed using the approximate Normal random variables are
\begin{equation}
\tXfnp = \tXfn + a(\tXfn)\, h + b(\tXfn)\, \sqrt{h}\, \tZ_n,
\label{eq:tf}
\end{equation}
and
\begin{equation}
\tXcnp = \tXcn + a(\tXcl)\, h + b(\tXcl)\, \sqrt{h}\, \tZ_n,
\label{eq:tc}
\end{equation}
and then finally we obtain
\[
(\hP_\ell {-} \hP_{\ell-1}) - (\tP_\ell {-} \tP_{\ell-1})
  \ =\ 
(f(\hXf_N) {-} f(\hXc_N)) -  (f(\tXf_N) {-} f(\tXc_N)).
\]

This description is for the case in which the timestep $h_\ell$ is halved
on each successive level.  There is a natural extension to other geometric
sequences such as $h_\ell = 4^{-\ell} h_0$, with the values of the coarse path 
drift and volatility being updated at the end of each coarse timestep, while
the fine path values are updated after each fine path timestep.

\section{Numerical analysis}

We begin by making two sets of assumptions which will be assumed to hold
throughout the analysis. 

The first concerns the drift and volatility functions, and assumes a 
greater degree of smoothness than the usual assumptions used to prove 
half-order strong convergence for the Euler-Maruyama discretisation
\cite{kp92}.
\begin{assumption}
\label{assumption:SDE}
The drift function $a: \RR \rightarrow \RR$ and 
volatility function $b: \RR \rightarrow \RR$ are both $C^1(\RR)$, and 
both they and their derivatives are Lipschitz continuous so that there 
exist constants $L_a$, $L_b$, $L'_a$, $L'_b$ such that
\[
\begin{array}{rrr}
|a(x){-}a(y)| \leq L_a |x{-}y|, &&
\ \, |b(x){-}b(y)| \leq L_b |x{-}y|, \\[0.05in]
|a'(x){-}a'(y)| \leq L'_a |x{-}y|, &&
|b'(x){-}b'(y)| \leq L'_b |x{-}y|.
\end{array}
\]
\end{assumption}

The second concerns approximate Normal random variables $\tZ$, 
and their relationship to corresponding exact Normal random 
variables $Z$.

\begin{assumption}
\label{assumption:Z}
Random variable pairs $(Z, \tZ)$ can be generated such that
$Z \sim N(0,1)$, $\EE[\tZ]\!=\!0$, and 
$\EE[\, |\tZ{-}Z|^p ] \leq \EE[\, |Z|^p]$ for all $p\!\geq\! 2$.
\end{assumption}
In most cases, the pairs will be generated as $(\Phi^{-1}(U), \tQ(U))$,
for a common uniform random variable $U$, but we leave open the 
possibility that they may be generated in an alternative way; 
this is needed later for Lemma \ref{lemma:first_diff_b}.
Note also that $|\tZ| \!\leq\! |Z| {+} |\tZ{-}Z|$, so a consequence 
of this assumption is that $\EE[\,|\tZ|^p] \leq 2^p\, \EE[\,|Z|^p]$;
this is also used later.

The proofs of the main results are greatly simplified by first proving 
the following lemma using the discrete Burkholder-Davis-Gundy 
inequality \cite{bdg72}.

\begin{lemma}
\label{lemma:key}
Suppose that for a time interval $[0,T]$ with $T {=} Nh$ and $t_n {=} nh$, 
we have the discrete equations
\[
D_{n+1} = D_n + A_n D_n h + B_n D_n h^{1/2} \tZ_n + \Theta_n h + \Psi_n h^{1/2},
\]
subject to initial data $D_0\!=\!0$, where $(Z_n,\tZ_n), n{\geq}0$
are i.i.d.~random variable pairs satisfying Assumption \ref{assumption:Z},
and using the standard filtration ${\cal F}_t$, 
$A_n, B_n, \Theta_n$ are ${\cal F}_{t_n}$-adapted, depending only
on $X_0$ and $(Z_m,\tZ_m), m{<}n$, whereas 
$\Psi_n$ is ${\cal F}_{t_{n+1}}$-adapted
with $\EE[\Psi_n | {\cal F}_{t_n}] \!=\! 0$.

Furthermore, suppose that $|A_n| \!\leq\! L_a$, $|B_n| \!\leq\! L_b$, 
where $L_a, L_b$ are as defined in Assumption \ref{assumption:SDE}, 
and for some $p\!\geq\! 2$ there are constants $c_1, c_2$, which do not 
depend on $h$, such that
\[
\max_{0\leq n < N} \EE[\, |\Psi_n|^p ] \leq c_1,~~~
\max_{0\leq n < N} \EE[\, |\Theta_n|^p ] \leq c_2.
\]
Then, there is a constant $c_3$ depending only on 
$L_a, L_b, p, T$ such that
\[
\EE \left[ \max_{0\leq n \leq N} |D_n|^p \right] \leq c_3 (c_1 {+} c_2).
\]
\end{lemma}

\noindent
{\bf Note}: the lemma includes as a special case the case in which
$\tZ_n \!=\! Z_n$, i.e.~with Normal random variables rather than approximate 
Normal random variables.

\begin{proof}
Summing over the first $n$ timesteps gives
\[
D_n = \sum_{m=0}^{n-1} 
\left\{  A_m D_m h + B_m D_m h^{1/2} \tZ_m + \Theta_m h + \Psi_m h^{1/2}\right\},
\]
and therefore if we define 
$\displaystyle E_n \equiv \EE\left[\max_{0< n'\leq n}|D_{n'}|^p \right]$ we obtain,
through Jensen's inequality,
\begin{eqnarray*}
E_n &\leq& 
~ 4^{p-1}\, \EE\left[ \,\max_{n'\leq n} \left| \sum_{m=0}^{n'-1}  A_m D_m h \right|^p\right]
+ 4^{p-1}\, \EE\left[ \,\max_{n'\leq n} \left| \sum_{m=0}^{n'-1}  B_m D_m h^{1/2} \tZ_m \right|^p\right]
\\ && \!\!\!\!
+\ 4^{p-1}\, \EE\left[ \,\max_{n'\leq n}  \left| \sum_{m=0}^{n'-1}  \Theta_m h \right|^p\right]
+ 4^{p-1}\, \EE\left[ \,\max_{n'\leq n} \left| \sum_{m=0}^{n'-1}  \Psi_m h^{1/2} \right|^p\right].
\end{eqnarray*}
For the first term, using Jensen's inequality again gives
\[
\EE\left[ \, \max_{n'\leq n} \left| \sum_{m=0}^{n'-1}  A_m D_m h \right|^p\ \right]
\ \leq\ h^p n^{p-1} \sum_{m=0}^{n-1} \EE\left[ \,| A_m D_m |^p \right]
\ \leq\  h\, T^{p-1} L_a^p \sum_{m=0}^{n-1} E_m.
\]
For the second term the discrete time Burkholder-Davis-Gundy 
inequality \cite{bdg72}, together with Jensen's inequality and
Assumption \ref{assumption:Z} gives the bound
\begin{eqnarray*}
\EE\left[ \, \max_{n'\leq n} \left| \sum_{m=0}^{n'-1}  B_m D_m h^{1/2} \tZ_m \right|^p\ \right]
&\leq& C_p\ \EE\left[ \,\left| \sum_{m=0}^{n-1} B^2_m D^2_m h \tZ_m^2 \right|^{p/2}\, \right]
\\
&\leq& C_p\ h^{p/2} n^{p/2-1} \sum_{m=0}^{n-1} \EE\left[ \,| B_m D_m|^p |\tZ_m|^p  \right]
\\
&\leq& C_p\, h\, T^{p/2-1} L_b^p\, 2^p\, \EE[\,|Z|^p] \sum_{m=0}^{n-1} E_m,
\end{eqnarray*}
with the constant $C_p$ depending only on $p$.
Similarly, the third term has the bound
\[
\EE\left[ \, \max_{n'\leq n} \left| \sum_{m=0}^{n'-1}  \Theta_m h \right|^p\ \right]
\ \leq\  T^p c_2,
\]
and the fourth term has the bound
\[
\EE\left[ \, \max_{n'\leq n} \left| \sum_{m=0}^{n'-1}  \Psi_m h^{1/2} \right|^p\ \right]
\ \leq\ C_p\,  T^{p/2}  \, c_1.
\]
Combining these four bounds we obtain
\begin{eqnarray*}
E_n &\leq& 4^{p-1} \left(T^{p-1} L_a^p  + C_p T^{p/2-1} L_b^p\, 2^p\, \EE[\,|Z|^p] \right)
\, h\, \sum_{m=0}^{n-1} E_m
\\ && +\ 4^{p-1} \left( T^p c_2 + C_p\, T^{p/2}\, c_1\right),
\end{eqnarray*}
and therefore by Gr\"onwall's inequality we obtain
\[
E_n \leq  4^{p-1} ( T^p\, c_2 + C_p\,  T^{p/2}\, c_1)\,
\exp\!\left(\rule{0in}{0.14in}4^{p-1} (T^p L_a^p  + C_p\, 2^p\, \EE[\,|Z|^p] \, T^{p/2} L_b^p)\right).
\]
Setting
$\displaystyle
c_3 = 4^{p-1} \max\left( T^p, C_p\, T^{p/2}\right) \ 
\exp\left(\rule{0in}{0.16in}4^{p-1} (T^p L_a^p  + C_p\, 2^p\, \EE[\,|Z|^p] \, T^{p/2} L_b^p)\right)
$
completes the proof.
\end{proof}

%\newpage

We are now ready to prove the first result, which is a generalisation 
of Lemma 2 in \cite{ghmr19}.

\begin{lemma}
\label{lemma:first_diff}
For a fixed time interval $T\!=\!N h$, for any $p\!\geq\!2$ there exists a 
constant $c$ which depends on $X_0, a, b, T, p$ but not on $h$ or 
$\EE[\, |\tZ{-}Z|^p]$, such that for any $0\!<\!n\!\leq\!N$
\[
\EE\left[ \max_{0\leq n\leq N}|\tXn{-}\hXn|^p\right] 
\leq c\ \EE[\, |\tZ{-}Z|^p]
\]
\end{lemma}
\begin{proof}
Taking the difference between 
\begin{equation}
  \tXnp = \tXn + a(\tXn)\, h + b(\tXn)\, h^{1/2} \tZ_n,
  \label{eq:original}
\end{equation}
and
\begin{equation}
  \hXnp = \hXn + a(\hXn)\, h + b(\hXn)\, h^{1/2} Z_n,
  \label{eq:approx}
\end{equation}
and defining $D_n \equiv \tXn {-} \hXn$, we obtain
\begin{equation}
D_{n+1} = D_n + a'(\xi_{1,n})\, D_n h + b'(\xi_{2,n}) D_n h^{1/2} \tZ_n
                                 + b(\hXn) \, h^{1/2} (\tZ_n{-}Z_n)
\label{eq:diff}
\end{equation}
for suitably defined $\xi_{1,n}$, $\xi_{2,n}$ arising from the Mean Value Theorem,
Lemma \ref{lemma:MVT}.

This is in the correct form to apply Lemma \ref{lemma:key} since
$|a'(\xi_{1,n})| \leq L_a$, ~
$|b'(\xi_{2,n})| \leq L_b$,
\[
\EE\left[b(\hXn) \, (\tZ_n{-}Z_n)\ |\ \hXn\right] \, =\, b(\hXn) \ \EE[\tZ_n{-}Z_n]\, =\, 0,
\]
and
\[
\EE\left[ |b(\hXn) \, (\tZ_n{-}Z_n)|^p \right]
\, =\, \EE[\, |b(\hXn)|^p] \ \EE[\,|\tZ_n{-}Z_n)|^p ].
\]
The result then follows from Lemma \ref{lemma:key} after noting that
the standard analysis of the Euler-Maruyama method (e.g.\ see \cite{kp92})
proves that $\EE[\, |\hXn|^p]$, $\EE[\, |a(\hXn)|^p]$ and $\EE[\, |b(\hXn)|^p]$ 
are all uniformly bounded on the time interval $[0,T]$.
\end{proof}

\begin{corollary}
\label{corollary:bounds}
For a fixed time interval $T\!=\!N h$, for any $p\!\geq\!2$ there exists a 
constant $c$ which depends on $X_0, a, b, T, p$ but not on $h$, 
such that for any $0\!<\!n\!\leq\!N$
\[
\EE[\, |\tXn|^p] \leq c, ~~~
\EE[\, |a(\tXn)|^p] \leq c, ~~~
\EE[\, |b(\tXn)|^p] \leq c.
\]
\end{corollary}
\begin{proof}
The standard analysis of the Euler-Maruyama method proves that $\EE[\, |\hXn|^p]$
is uniformly bounded on $[0,T]$, and so it follows from Lemma \ref{lemma:first_diff} 
that there exist constants $c_1$, $c_2$ such that
\[
\EE[\, |\tXn|^p] 
\ \leq\ c_1 + c_2\, \EE[\, |\tZ{-}Z|^p]
\ \leq\ c_1 + c_2\, \EE[\, |Z|^p ]
\]
due to Assumption \ref{assumption:Z}.
Since $|a(\tXn)| \leq |a(0)| + L_a |\tXn|$, it follows that
\[
|a(\tXn)|^p \leq 2^{p-1} \left( |a(0)|^p + L^p_a\, |\tXn|^p \right),
\]
and therefore $\EE[\, |a(\tXn)|^p]$ can be uniformly bounded, and similarly
$\EE[\, |b(\tXn)|^p]$.
\end{proof}

\begin{corollary}
\label{corollary:delta_t}
For a fixed time interval $T\!=\!N h$, for any $p\!\geq\!2$ there exists a 
constant $c$ which depends on $X_0, a, b, T, p$ but not on $h$ or 
$\EE[\, |\tZ{-}Z|^p]$ such that
\[
\max_{0\leq n < N} \EE[\, |\tXnp{-}\tXn|^p] \leq c\, h^{p/2}, ~~~
\max_{0\leq n < N} \EE[\, |\hXnp{-}\hXn|^p] \leq c\, h^{p/2},
\]
and
\[
\max_{0\leq n < N} \EE\left[\, |(\tXnp{-}\tXn)-(\hXnp{-}\hXn)|^p\right] 
\leq c\, h^{p/2}\, \EE[\, |\tZ{-}Z|^p].
\]
\end{corollary}
\begin{proof}
Since
\[
\EE[\, |\tXnp{-}\tXn|^p]
\ \leq\ 2^{p-1} \left( \EE[\,|a(\tXn)|^p] \, h^p + \EE[\,|b(\tXn)|^p]\, h^{p/2} \EE[\,|\tZ_n|^p] \right),
\]
the first assertion follows from the uniform boundedness of $\EE[\,|a(\tXn)|^p]$ and $\EE[\,|b(\tXn)|^p]$,
together with the trivial inequality $h^p\!\leq\! h^{p/2}T^{p/2}$ and the $\EE[\,|\tZ_n|^p]$ bound due
to Assumption \ref{assumption:Z}. The second assertion follows similarly.

Re-arranging Equation (\ref{eq:diff}) gives
\[
D_{n+1}-D_n = a'(\xi_{1,n})\, D_n h + b'(\xi_{2,n}) D_n h^{1/2} \tZ_n
                                 + b(\hXn) \, h^{1/2} (\tZ_n{-}Z_n),
\]
with $D_n \equiv \tXn{-}\hXn$. Hence,
\begin{eqnarray*}
  \lefteqn{  \EE[\, |(\tXnp{-}\tXn)-(\hXnp{-}\hXn)|^p]}
\\ &\leq& 3^{p-1}\left(
    L_a^p\, \EE[\, |D_n|^p]\, h^p
    + L_b^p\, \EE[\, |D_n|^p]\, h^{p/2}\, \EE[\,|\tZ_n|^p]
    + \EE[\, |b(\hXn)|^p]\, h^{p/2}\, \EE[\,|\tZ{-}Z|^p] \right),
\end{eqnarray*}
and the third assertion follows from the bounds on $\EE[\, |\tXn{-}\hXn|^p]$,
and $\EE[\, |b(\hXn)|^p]$.
\end{proof}

We now have the first MLMC results involving level $\ell$ fine paths
and level $\ell{-}1$ coarse paths, as defined in Equations
(\ref{eq:hf})--(\ref{eq:tc}).

\begin{lemma}
\label{lemma:first_diff_b}
For a fixed time interval $T\!=\!N h$, for any $p\!\geq\!2$ there exists a 
constant $c$ which depends on $X_0, a, b, T, p$, but not on $h$ or 
$\EE[\, |\tZ{-}Z|^p]$, such that for any $0\!<\!n\!\leq\!N$
\[
\EE\left[ \max_{0< n \leq N} |\tXfn{-}\hXfn|^p\right] \leq c\, \EE[\, |\tZ{-}Z|^p], ~~~ 
\EE\left[ \max_{0< n \leq N} |\tXcn{-}\hXcn|^p\right] \leq c\, \EE[\, |\tZ{-}Z|^p].
\]
\end{lemma}
\begin{proof}
The first assertion comes immediately from Lemma \ref{lemma:first_diff}, 
but the second assertion requires the observation that if we set
\[
 Z_3  = (  Z_1{+}  Z_2)/\sqrt{2}, ~~~~
\tZ_3 = (\tZ_1{+}\tZ_2)/\sqrt{2}, 
\]
where the independent pairs $(Z_1, \tZ_1)$ and $(Z_2, \tZ_2)$ satisfy 
Assumption \ref{assumption:Z}, then $Z_3 \sim N(0,1)$, $\EE[\tZ_3]\!=\!0$, 
and by Jensen's inequality
\[
\EE\left[\, | \tZ_3 {-} Z_3 |^p \right] \ \leq\ 
2^{p/2-1} \left( \EE[\, |\tZ_1{-}Z_1|^p] + \EE[\, |\tZ_2{-}Z_2|^p] \right)
\ \leq\ 2^{p/2\,} \EE[\, |Z|^p].
\]
Therefore the pair $(Z_3, \tZ_3)$ also satisfies Assumption 
\ref{assumption:Z} apart from an increased bound on 
$\EE\left[\, | \tZ_3 {-} Z_3 |^p \right]$.  This requires minor 
changes to the constants in the subsequent lemmas, but in the end 
the desired result follows from Lemma \ref{lemma:first_diff}.
\end{proof}

\begin{lemma}
\label{lemma:delta_h}
For a fixed time interval $T\!=\!N h$, for any $p\!\geq\!2$ there exists a 
constant $c$ which depends on $X_0, a, b, T, p$, but not on $h$ or 
$\EE[\, |\tZ{-}Z|^p]$, such that for any $0\!<\!n\!\leq\!N$
\[
\EE\left[ \max_{0< n \leq N} |\tXfn{-}\tXcn|^p\right] \leq c\, h^{p/2}, ~~~ 
\EE\left[ \max_{0< n \leq N} |\hXfn{-}\hXcn|^p\right] \leq c\, h^{p/2}, ~~~ 
\]
\end{lemma}
\begin{proof}
Defining $D_n \equiv \tXfn {-} \tXcn$, and taking the difference between 
equations (\ref{eq:tf}) and (\ref{eq:tc}) gives
\begin{eqnarray*}
D_{n+1} &=& D_n + (a(\tXfn)-a(\tXcn))\, h + (b(\tXfn)-b(\tXcn)) \, h^{1/2} \tZ_n
\\     &&~~~~ + (a(\tXcn)-a(\tXcl)) \, h + (b(\tXcn)-b(\tXcl)) \, h^{1/2} \tZ_n
\\     &=& D_n + a'(\xi_{1,n})\, D_n h + b'(\xi_{2,n}) D_n h^{1/2} \tZ_n
\\     &&~~~~ + a'(\xi_{3,n})\, (\tXcn{-}\tXcl) \, h
          +  b'(\xi_{4,n})\, (\tXcn{-}\tXcl) \, h^{1/2} \tZ_n
\end{eqnarray*}
for suitably defined $\xi_{1,n}$, $\xi_{2,n}, \xi_{3,n}$, $\xi_{4,n}$ arising 
from the Mean Value Theorem.  
Noting that
\[
\EE[\,|a'(\xi_{3,n})\, (\tXcn{-}\tXcl) |^p] \leq
L_a^p \ \EE[\,|\tXcn{-}\tXcl|^p],
\]
and
\[
\EE[\,|b'(\xi_{4,n})\, (\tXcn{-}\tXcl)  \tZ_n|^p] 
\ \leq\ L_b^p \ \EE[\,|\tXcn{-}\tXcl|^p]\ \EE[\,|\tZ|^p]
\ \leq\ L_b^p\, 2^{p} \, \EE[\,|Z|^p]\ \EE[\,|\tXcn{-}\tXcl|^p],
\]
the first assertion then follows again from 
Lemma \ref{lemma:key} after using the bounds for $\EE[\,|\tXcn{-}\tXcl|^p]$ 
which come from Corollary \ref{corollary:delta_t}.

The second assertion follows similarly.
\end{proof}

We now come to the analysis of the cross-difference.

\begin{lemma}
\label{lemma:second_diff}
For a fixed time interval $T\!=\!N h$, for any $p, q$ with 
$2\!\leq\! p \!<\! q$ there exists a constant $c$ which depends on $
X_0, a, b, T$ but not on $h$ or $\EE[\, |\tZ{-}Z|^q]$, such that
\[
\EE\left[ \max_{0\leq n < N} \left|\tXfn{-}\tXcn {-} \hXfn{+}\hXcn\right|^p\right] 
\leq c\, h^{p/2} \left( \EE[\,|\tZ{-}Z|^q] \right)^{p/q}.
\]
\end{lemma}

\begin{proof}
Defining $D_n \equiv \tXfn{-}\tXcn {-} \hXfn{+}\hXcn$, 
then the difference of Equations (\ref{eq:hf})--(\ref{eq:tc}) 
together with Lemma \ref{lemma:MVT2} gives
\begin{eqnarray*}
\lefteqn{D_{n+1}}
\\ &=& D_n
+ \left( a(\tXfn) {-} a(\tXcn) {-} a(\hXfn) {+} a(\hXcn) \right) h
+ \left( b(\tXfn) {-} b(\tXcn) {-} b(\hXfn) {+} b(\hXcn) \right) h^{1/2} \tZ_n
\\ && ~~~~
+ \left( a(\tXcn) {-} a(\tXcl) {-} a(\hXcn) {+} a(\hXcl) \right) h
+ \left( b(\tXcn) {-} b(\tXcl) {-} b(\hXcn) {+} b(\hXcl) \right) h^{1/2} \tZ_n
\\ && \hspace{3.4in}
+ \left( b(\hXfn) {-} b(\hXcl) \right) h^{1/2} (\tZ_n {-} Z_n)
\\ &=& D_n
+ a'(\xi_{1,n})\, D_n\, h
+ b'(\xi_{2,n})\, D_n\, h^{1/2} \tZ_n
\\ && ~~~~
+ a'(\xi_{3,n}) \left( \tXcn {-} \tXcl {-} \hXcn {+} \hXcl \right) h
+ b'(\xi_{4,n}) \left( \tXcn {-} \tXcl {-} \hXcn {+} \hXcl \right) h^{1/2} \tZ_n
\\ && ~~~~ +\ (R_{1,n} +  R_{3,n})\, h + (R_{2,n} + R_{4,n})\, h^{1/2} \tZ_n
+ b'(\xi_{5,n}) \left( \hXfn {-} \hXcl \right) h^{1/2} (\tZ_n {-} Z_n)
\end{eqnarray*}
for suitably defined $\xi_{1,n}, \xi_{2,n}, \xi_{3,n}$, $\xi_{4,n}, \xi_{5,n}$ 
arising from Lemma \ref{lemma:MVT2} and the Mean Value Theorem, and with
\begin{eqnarray*}
|R_{1,n}| &\leq& \fracs{1}{2} L'_a
 \left( |\tXfn {-} \tXcn| + |\hXfn {-} \hXcn| \right)
 \left( |\tXfn {-} \hXfn| + |\tXcn {-} \hXcn| \right) \\
|R_{2,n}| &\leq& \fracs{1}{2} L'_b
 \left( |\tXfn {-} \tXcn| + |\hXfn {-} \hXcn| \right)
 \left( |\tXfn {-} \hXfn| + |\tXcn {-} \hXcn| \right) \\
|R_{3,n}| &\leq& \fracs{1}{2} L'_a
 \left( |\tXcn {-} \tXcl| + |\hXcn {-} \hXcl| \right)
 \left( |\tXcn {-} \hXcn| + |\tXcl {-} \hXcl| \right) \\
|R_{4,n}| &\leq& \fracs{1}{2} L'_b
 \left( |\tXcn {-} \tXcl| + |\hXcn {-} \hXcl| \right)
 \left( |\tXcn {-} \hXcn| + |\tXcl {-} \hXcl| \right).
\end{eqnarray*}
This equation is in the correct form for the application of Lemma \ref{lemma:key}
with
\begin{eqnarray*}
  \Theta_n &=&   a'(\xi_{3,n}) \left( \tXcn {-} \tXcl {-} \hXcn {+} \hXcl \right)
+ (R_{1,n} +  R_{3,n}), \\
  \Psi_n &=&  b'(\xi_{4,n}) \left( \tXcn {-} \tXcl {-} \hXcn {+} \hXcl \right) \tZ_n
+ (R_{2,n} + R_{4,n})\, \tZ_n\\
&& +\ b'(\xi_{5,n}) \left( \hXfn {-} \hXcl \right) (\tZ_n {-} Z_n).
\end{eqnarray*}

Corollary \ref{corollary:delta_t} and Lemma \ref{lemma:first_diff} together 
with the H\"older inequality imply that there exists a constant $c$ such that
\begin{eqnarray*}
\EE[\, |\tXfn {-} \tXcn|^p |\tXfn {-} \hXfn|^p ]
&\leq& \left( \EE[\,|\tXfn {-} \tXcn|^{p/(1-p/q)}] \right)^{1 - p/q}
       \left( \EE[\,|\tXfn {-} \hXfn|^q] \right)^{p/q}
\\  &\leq& c\, h^{p/2} \left(\EE[\, |\tZ{-}Z|^q]\right)^{p/q}.
\end{eqnarray*}
Bounding the other terms similarly, there is a different constant $c$ such that
\[
\EE[\, |R_{i,n}|^p] \leq  c\ h^{p/2} \left(\,\EE[\,|\tZ{-}Z|^q] \right)^{p/q}, ~~~ i = 1, 2, 3, 4.
\]
Due to Corollary \ref{corollary:delta_t} we also have
\[
\EE[ \, |\tXcn {-} \tXcl {-} \hXcn {+} \hXcl |^p ] 
\leq c\, h^{p/2} \left(\EE[\,|\tZ{-}Z|^q] \right)^{p/q}
\]
for some constant $c$, and finally, for another constant $c$,
\[
\EE[ \, |  ( \hXfn {-} \hXcl ) (\tZ_n {-} Z_n)|^p ]
= \EE[ \, | \hXfn {-} \hXcl|^p] ~ \EE[\, |\tZ {-} Z|^p]
\leq c \, h^{p/2}\, \left( \EE[\,|\tZ{-}Z|^q] \right)^{p/q}.
\]

Hence, we end up concluding that there exists another constant $c$ such that
\[
\EE[\, |\Psi_n|^p ] \leq c\, h^{p/2} \left( \EE[\,|\tZ{-}Z|^q] \right)^{p/q}, ~~~
\EE[\, |\Theta_n|^p ] \leq c\, h^{p/2} \left( \EE[\,|\tZ{-}Z|^q] \right)^{p/q},
\]
and then Lemma \ref{lemma:key} gives us the desired final result.
\end{proof}

We now obtain a lemma for output functions $f(x)$ which are locally
Lipschitz with at worst a polynomial growth as $|x|\rightarrow\infty$.

\begin{lemma}
\label{lemma:polynomial}
For a fixed time interval $T\!=\!N h$, if the function $f: \RR \rightarrow \RR$
is $C^1(\RR)$ and there is an exponent $r>0$ and constants $L_f, L'_f$ such that
\[
|f(x)-f(y)| \leq L_f\, (1+|x|^r+|y|^r)\ |x{-}y|, ~~~
|f'(x)-f'(y)| \leq L'_f\, (1+|x|^r+|y|^r)\ |x{-}y|,
\]
then for any $q\!>\!2$ there exists a constant $c$ which depends on 
$X_0, a, b, f, T$, but not on $h$ or $\EE[\, |\tZ{-}Z|^q]$, such that
\[
\VV\left[ f(\hXf_N){-}f(\hXc_N) {-} f(\tXf_N){+}f(\tX_Nc)\right] 
\leq c\, h \left(\EE[\, |\tZ{-}Z|^q]\right)^{2/q}.
\]
\end{lemma}
\begin{proof}
Given the assumptions on $f$, we are able to follow the proof 
of Lemma \ref{lemma:MVT2} to obtain
\[
f(\hXf_N){-}f(\hXc_N) {-} f(\tXf_N){+}f(\tXc_N)
= f'(\xi)\ (\hXf_N{-}\hXc_N{-}\tXf_N{+}\tXc_N)
+ R,
\]
where
\[
|f'(\xi)| \leq  L_f \left(1 + |\hXf_N|^r + |\hXc_N|^r + |\tXf_N|^r + |\tXc_N|^r \right)
\]
and
\begin{eqnarray*}
|R| &\leq& \fracs{1}{2} L'_f \left(1 + |\hXf_N|^r + |\hXc_N|^r + |\tXf_N|^r + |\tXc_N|^r \right)
\\ && \times \  \left( |\hXf_N{-}\hXc_N| + |\tXf_N{-}\tXc_N| \right)\, 
                \left( |\hXf_N{-}\tXf_N| + |\hXc_N{-}\tXc_N| \right).
\end{eqnarray*}
Hence,
\[
\VV\left[ f(\hXf_N){-}f(\hXc_N) {-} f(\tXf_N){+}f(\tXc_N)\right] \leq 
2\, \EE\left[\, (f'(\xi))^2|\hXf_N{-}\hXc_N{-}\tXf_N{+}\tXc_N|^2\right]
+ 2\, \EE[R^2].
\]
Due to H\"older's inequality, 
\begin{eqnarray*}
\lefteqn{
\EE\left[\, |\hXf_N|^{2r} \, |\hXf_N{-}\hXc_N|^2 \, |\hXf_N{-}\tXf_N|^2 \right] 
}
\\ &\leq& 
\left( \EE[\,|\hXf_N|^{2r/(1/2 - 1/q)}]\right)^{1/2 - 1/q}
\left( \EE[\,|\hXf_N{-}\hXc_N|^{2/(1/2 - 1/q)}]\right)^{1/2 - 1/q}
\left( \EE[\,|\hXf_N{-}\tXf_N|^q] \right)^{2/q}
\end{eqnarray*}
Note that $\EE[ \, |\hXf_N|^{2r/(1/2 - 1/q)}]$ is finite and uniformly bounded due to 
Corollary \ref{corollary:bounds}.
The other terms in $\EE[R^2]$ can be bounded similarly, and therefore due to the bounds
from Lemmas \ref{lemma:first_diff} and \ref{lemma:delta_h} there exists a constant $c$,
not depending on $h$ or $\EE[\,|\tZ{-}Z|^q]$, such that
\[
\EE[R^2] \leq c\, h \left(\EE[\,|\tZ{-}Z|^q]\right)^{2/q}.
\]
Similarly, choosing $p$ such that $2\!<\!p\!<\!q$, then due to H\"older's inequality, 
\begin{eqnarray*}
\lefteqn{
\EE\left[\, (f'(\xi))^2|\hXf_N{-}\hXc_N{-}\tXf_N{+}\tXc_N|^2\right]
}
\\ &\leq&
\left(\EE[ \, |f'(\xi)|^{2/(1-2/p)}] \right)^{1-2/p}
\left(\EE\left[\,|\hXf_N{-}\hXc_N{-}\tXf_N{+}\tXc_N|^p\right] \right)^{2/p}
\end{eqnarray*}
$\EE[ \, |f'(\xi)|^{2/(1-2/p)}]$ is finite and uniformly bounded due to 
Corollary \ref{corollary:bounds}, and therefore the bound for 
$\EE[\, |\hXf_N{-}\hXc_N{-}\tXf_N{+}\tXc_N|^p]$ from
Lemma \ref{lemma:second_diff} completes the proof.
\end{proof}

In finance applications, put and call options correspond to
$f(x)\equiv \max(K{-}x,0)$ and $\max(x{-}K,0)$, respectively,
with $K{>}0$ being the ``strike''. 
More generally, we can consider functions $f$ which are 
globally Lipschitz with a derivative which exists and is 
continuous everywhere except at a single point $K$.

Heuristically, the four values $\hXf_N, \hXc_N, \tXf_N, \tXc_N$ 
do not differ from each other, or from $X_T$, by more than 
$O(\, \max\{ h^{1/2}, \, (\EE[\, |\tZ{-}Z|^2])^{1/2}\}\, )$.
If $X_T$ has a bounded probability density, then the 
probability that $X_T$ is within this distance of $K$
is $O(\, \max\{ h^{1/2}, \, (\EE[\, |\tZ{-}Z|^2])^{1/2}\}\, )$, 
and in this case the global Lipschitz property for $f$ gives
\[
f(\hXf_N) {-} f(\hXc_N) {-} f(\tXf_N) {+} f(\tXc_N)
= O\left( \min\left\{ h^{1/2}, \, (\EE[\, |\tZ{-}Z|^2])^{1/2} \right\} \right),
\]
with the first term in the minimum coming from 
\[
f(\hXf_N) {-} f(\hXc_N) {-} f(\tXf_N) {+} f(\tXc_N)
= \left( f(\hXf_N) {-} f(\hXc_N) \right) - \left(f(\tXf_N) {+} f(\tXc_N)\right)
\]
together with Lemma \ref{lemma:delta_h}, while the second term comes from
\[
f(\hXf_N) {-} f(\hXc_N) {-} f(\tXf_N) {+} f(\tXc_N)
= \left( f(\hXf_N) {-} f(\tXf_N) \right) - \left(f(\hXc_N) {+} f(\tXc_N)\right)
\]
together with Lemma \ref{lemma:first_diff}.

On the other hand, if $X_T$ is more than this distance from $K$ 
then all four values will be on the same side of $K$ and then
\[
f(\hXf_N) {-} f(\hXc_N) {-} f(\tXf_N) {+} f(\tXc_N)
= O\left( h^{1/2}\, (\EE[\, |\tZ{-}Z|^2])^{1/2} \right).
\]
Consequently,
\begin{eqnarray*}
\lefteqn{
\VV[f(\hXf_N) {-} f(\hXc_N) {-} f(\tXf_N) {+} f(\tXc_N)] 
}
\\ &=& O\left(\max\left\{ h^{1/2},\, (\EE[\, |\tZ{-}Z|^2])^{1/2}\right\} \right)
 \times O\left( \min\left\{ h,\, \EE[\, |\tZ{-}Z|^2] \right\} \right)
 + O\left( h \ \EE[\, |\tZ{-}Z|^2] \right)
\\ &=& O\left( \min\left\{ h\, (\EE[\, |\tZ{-}Z|^2])^{1/2},
                       h^{1/2}\, \EE[\, |\tZ{-}Z|^2] \right\} \right).
\end{eqnarray*}

%The numerical results to be presented later are consistent with this 
%heuristic analysis.
The bound in the following lemma (for the case $q{\approx}2$) is
slightly weaker, but the proof follows along similar lines in
establishing that the dominant contribution to the variance comes
from samples with $X_T$ near $K$.

In the proof, we will use the notation
\[
g_1(h, \EE[\, |\tZ{-}Z|^q]) \prec g_2(h, \EE[\, |\tZ{-}Z|^q])
\]
for any two strictly positive functions $g_1, g_2$ to mean that there
exists a constant $c\!>\!0$ which does not depend on $h$ or
$\EE[\, |\tZ{-}Z|^q]$ such that
\[
g_1(h, \EE[\, |\tZ{-}Z|^q]) < c\ g_2(h, \EE[\, |\tZ{-}Z|^q]).
\]
Note that if $0\!<\!a\!<\!b$ then
\[
h^b \leq T^{b-a} h^a
~~~ \Longrightarrow~~
h^b \prec h^a,
\]
and likewise, due to Assumption \ref{assumption:Z}, 
\[
\left(\EE[\, |\tZ{-}Z|^q]\right)^b \leq 
\left(\EE[\, |Z|^q] \right)^{b-a} \left(\EE[\, |\tZ{-}Z|^q]\right)^a
~~~ \Longrightarrow~~
\left(\EE[\, |\tZ{-}Z|^q]\right)^b \prec
\left(\EE[\, |\tZ{-}Z|^q]\right)^a.
\]

\begin{lemma}
\label{lemma:polynomial_2}
Suppose that the conditions of Lemma \ref{lemma:polynomial} are slightly
modified so that $f: \RR \rightarrow \RR$ is $C^1(\RR\backslash K)$ and 
there is an exponent $r\!>\!0$ and constants $L_f, L'_f$ such
\begin{eqnarray*}
|f(x)-f(y)| \leq L_f (1+|x|^r+|y|^r)\, |x-y|, && \mbox{for all } x, y\\
|f'(x)-f'(y)| \leq L'_f (1+|x|^r+|y|^r)\, |x-y|, && 
\mbox{if \underline{either}\ } x>y>K \mbox{\ \underline{or}\ } x<y<K,
\end{eqnarray*}
and furthermore $X_T$ has a bounded probability density in the 
neighbourhood of $K$ and therefore there is a constant $c_\rho\!>\!0$
such that for any $D\!>\!0$
\[
\PP[ \, |X_T{-}K| < D ] \leq c_\rho\, D.
\]

Then for any $q\!>\!2$ and any $\delta\!>\!0$,
there exists a constant $c_\delta$ which depends on $X_0$, $a$, $b$, 
$f$, $T$, $q$ 
and $\delta$, but not on $h$ or $\EE[\, |\tZ{-}Z|^q]$, such that
\begin{eqnarray}
\lefteqn{
\hspace{-0.5in}
\VV\left[ f(\hXf_N) {-} f(\hXc_N) {-} f(\tXf_N) {+} f(\tXc_N) \right] 
}
\nonumber \\ &\leq& c_\delta \, \min\left\{ 
h \left(\EE[\, |\tZ{-}Z|^q]\right)^{(1-\delta)/(q+1)}\!,\, 
h^{(1-\delta)/2-1/q} \left(\EE[\, |\tZ{-}Z|^q]\right)^{2/q} \right \}
\label{eq:bound}
\end{eqnarray}
\end{lemma}
\begin{proof}
The proof follows an approach used previously in the analysis of MLMC variance
for similar options in the context of multi-dimensional SDEs (Theorem 5.2 in \cite{gs14}).

The proof is given for $0\!<\!\delta\!< 1{-}1/q$.  If the assertion 
is true for $\delta$ in this range then it also holds for larger values.

If we define the events $A$ and $B$ as
\[
A: | X_T - K | \leq D, ~~~~
B: \max\left\{
 |\hXf_N {-} X_T|,  |\hXc_N {-} X_T|,  |\tXf_N {-} \hXf_N|,  |\tXc_N {-} \hXc_N|, 
 \right\} \geq D/2,
\]
for some choice of constant $D\!>\!0$, and define
$\Delta f \equiv  f(\hXf_N) {-} f(\hXc_N) {-} f(\tXf_N) {+} f(\tXc_N)$, then
\[
\VV\left[ \Delta f \right] 
\leq \EE\left[ (\Delta f)^2 \one_{A\cup B} \right]
  +  \EE\left[ (\Delta f)^2 \one_{A^c\cap B^c} \right]
\]
where $A^c, B^c$ are the complements of $A$ and $B$, and $\one_C$ is the indicator function
which has value $1$ if the random sample $\omega\in C$, and 0 otherwise.
Note that if $\omega \!\in\! A^c\!\cap\! B^c$ then the four values 
$\hXf_N, \hXc_N, \tXf_N, \tXc_N$  are all on the same side of $K$, and therefore
the proof in Lemma \ref{lemma:polynomial} means that 
\begin{equation}
\EE\left[ (\Delta f)^2 \one_{A^c\cap B^c} \right]
\prec h \left(\EE[\, |\tZ{-}Z|^q]\right)^{2/q}
\prec h \left(\EE[\, |\tZ{-}Z|^q]\right)^{(1-\delta)/(q+1)}
\label{eq:bound1}
\end{equation}
so $\EE\left[ (\Delta f)^2 \one_{A^c\cap B^c} \right]$ is not the dominant
contributor to the bound in (\ref{eq:bound}).

To address the other term, $\EE\left[ (\Delta f)^2 \one_{A\cup B} \right]$
we begin by noting that the two terms in the bound on the r.h.s.~of
(\ref{eq:bound}) are equal when
$h^{1/2} = (\EE[\, |\tZ{-}Z|^q])^{1/(q+1)}$.

\vspace{0.2in}

\underline{Case A}: $h^{1/2} \leq (\EE[\, |\tZ{-}Z|^q])^{1/(q+1)}$.

In this case, we set $D\!=\!(\EE[\, |\tZ{-}Z|^q])^{(1-\delta/2)/(q+1)}$,
and by H\"older's inequality we have
\[
\EE\left[ (\Delta f)^2 \one_{A\cup B} \right]
\leq \left( \EE[ |\Delta f|^{2/\delta'}] \right)^{\delta'}
     \left( \EE[\one_{A\cup B}] \right)^{1-\delta'}
\leq \left( \EE[ |\Delta f|^{2/\delta'}] \right)^{\delta'}
     \left( \PP[A] + \PP[B] \right)^{1-\delta'},
\]
where $\delta'=\delta/(2{-}\delta)$ so that
$1{-}\delta'=(1{-}\delta)/(1{-}\delta/2)$.

Due to the assumed bounded density for $X_T$, we have
$\PP[A] \prec D$.  Also, 
\begin{eqnarray*}
\PP[B] &\leq&  ~
\PP[\, |\hXf_N {-} X_T| > D/2] +
\PP[\, |\hXc_N {-} X_T| > D/2]
\\ & & \!\!\!\! +\
\PP[\, |\tXf_N {-} \hXf_N| > D/2] +
\PP[\, |\tXc_N {-} \hXc_N| > D/2].
\end{eqnarray*}
By the Markov inequality, together with the standard strong convergence
results,
\[
\PP[\, |\hXf_N {-} X_T| \!>\! D/2] 
\leq \frac{\EE[\, |\hXf_N {-} X_T|^p]}{(D/2)^p}
\prec \frac{h^{p/2}}{D^p}
\prec (\EE[\, |\tZ{-}Z|^q])^{p\delta/(2q+2)}
\prec D,
\]
by choosing $p>2/\delta - 1$.  A similar bound follows for
$\PP[\, |\hXc_N {-} X_T| \!>\! D/2]$.
In addition, the Markov inequality, together with Lemma \ref{lemma:first_diff}, gives
\[
\PP[\, |\tXf_N {-} \hXf_N| \!>\! D/2] 
\leq \frac{\EE[\, |\tXf_N {-} \hXf_N|^q]}{D^q}
\prec \frac{\EE[\, |\tZ{-}Z|^q]}{D^q} 
\prec D,
\]
and a similar bound holds for $\PP[\, |\tXc_N {-} \hXc_N| \!>\! D/2]$.
The conclusion from this is that $\PP[B]\prec D$.

Hence,
\[
\left( \PP[A] + \PP[B] \right)^{(1-\delta)/(1-\delta/2)} \prec D^{(1-\delta)/(1-\delta/2)}
 =  (\EE[\, |\tZ{-}Z|^q])^{(1-\delta)/(q+1)}.
\]

In addition, we have
\[
|\Delta f|^{2/\delta'} \leq 2^{2/\delta'-1}\left(
  |f(\hXf_N) {-} f(\hXc_N)|^{2/\delta'} + |f(\tXf_N) {-} f(\tXc_N)|^{2/\delta'}
  \right),
\]
and due to H\"older's inequality and the bounds in Corollary \ref{corollary:bounds}
and Lemma \ref{lemma:delta_h} we have
\begin{eqnarray*}
\EE[\,|f(\tXf_N) {-} f(\tXc_N)|^{2/\delta'}]
&\leq& L_f^{2/\delta'} \left( \EE[\, |1 + c |\tXf_N|^r + c |\tXc_N|^r|^{4/\delta'}]\right)^{1/2}
       \left( \EE[\, |\tXf_N {-} \tXc_N|^{4/\delta'}]\right)^{1/2}
       \\ &\prec& h^{1/\delta'}.
\end{eqnarray*}
There is a similar bound for $\EE[\,|f(\hXf_N) {-} f(\hXc_N)|^{2/\delta'}]$
and hence we have the result that 
$\EE\left[ (\Delta f)^2 \one_{A\cup B} \right]
\prec h \, (\EE[\, |\tZ{-}Z|^q])^{(1-\delta)/(q+1)}$
when $h^{1/2}\leq  (\EE[\, |\tZ{-}Z|^q])^{1/(q+1)}$.

\vspace{0.2in}

\underline{Case B}: $h^{1/2} \geq (\EE[\, |\tZ{-}Z|^q])^{1/(q+1)}$.

In this case we set $D\!=\!h^{(1-\delta)/2}$ and by H\"older's
inequality we have
\begin{eqnarray*}
\EE\left[ (\Delta f)^2 \one_{A\cup B} \right]
&\leq& \left( \EE[ |\Delta f|^{2/(2/q+\delta')}] \right)^{2/q +\delta'}
       \left( \EE[\one_{A\cup B}] \right)^{1-2/q-\delta'}
\\ &\leq& \left( \EE[ |\Delta f|^{2/(2/q+\delta')}] \right)^{2/q+\delta'}
          \left( \PP[A] + \PP[B] \right)^{1-2/q-\delta'},
\end{eqnarray*}
where $\delta'= 2\delta/(q(1{-}\delta))$ so that 
$(1{-}2/q{-}\delta')(1{-}\delta)/2 = (1{-}\delta)/2 - 1/q$.

We again have $\PP[A] \prec D$.  By the Markov inequality,
together with the standard strong convergence results,
\[
\PP[\, |\hXf_N {-} X_T| \!>\! D/2] 
\leq \frac{\EE[\, |\hXf_N {-} X_T|^p]}{(D/2)^p}
\prec \frac{h^{p/2}}{h^{p(1-\delta)/2}}
= h^{p\delta/2}
\prec D,
\]
by choosing $p\!>\!1/\delta$, and a similar bound follows for
$\PP[\, |\hXc_N {-} X_T| \!>\! D/2]$.
In addition, the Markov inequality, together with Lemma \ref{lemma:first_diff}, gives
\[
\PP[\, |\tXf_N {-} \hXf_N| \!>\! D/2] 
\leq \frac{\EE[\, |\tXf_N {-} \hXf_N|^q]}{D^q}
\prec \frac{\EE[\, |\tZ{-}Z|^q]}{D^q} 
\prec \frac{h^{(q+1)/2}}{h^{q(1-\delta)/2}}
= h^{1/2+q\delta/2}
\prec D,
\]
and a similar bound holds for $\PP[\, |\tXc_N {-} \hXc_N| \!>\! D/2]$.
The conclusion from this is that $\PP[B]\prec D$, as before, and so
\[
\left( \PP[A] + \PP[B] \right)^{1-2/q-\delta'}
\prec D^{1-2/q-\delta'}
=  h^{(1-2/q-\delta')(1-\delta)/2}
= h^{(1-\delta)/2 - 1/q}.
\]

In addition, defining $\delta''=\delta'/(2/q{+}\delta')$ so that
$2/(2/q+\delta')=q(1{-}\delta'')$, we have
\[
|\Delta f|^{q(1{-}\delta'')} \leq 2^{q(1{-}\delta'')-1}\left(
  |f(\tXf_N) {-} f(\hXf_N)|^{q(1{-}\delta'')} + |f(\tXc_N) {-} f(\hXc_N)|^{q(1{-}\delta'')}
  \right)
\]
and due to H\"older's inequality and the bounds in Corollary \ref{corollary:bounds}
and Lemma \ref{lemma:delta_h} we have
\begin{eqnarray*}
\lefteqn{\hspace*{-0.5in}
\EE[\, |f(\tXf_N) {-} f(\hXf_N)|^{q(1-\delta'')}]
}
\\ \hspace*{0.5in}
  &\leq& L_f^{q(1-\delta'')} \left( \EE[\, |1 {+} c |\tXf_N|^r {+} c |\hXf_N|^r|^{q(1-\delta'')/\delta''}]\right)^{\delta''}
  \!     \left( \EE[\, |\tXf_N {-} \hXf_N|^q]\right)^{1-\delta''}
\\ &\prec&  (\EE[\, |\tZ{-}Z|^q])^{1-\delta''},
\end{eqnarray*}
\[
\hspace*{-0.5in}  
\Longrightarrow ~~ \left( \EE[\,|f(\tXf_N) {-} f(\hXf_N)|^{2/(2/q+\delta')}] \right)^{2/q+\delta'}
\prec\ (\EE[\, |\tZ{-}Z|^q])^{2/q}.
\]
There is a similar bound for $\EE[\, |f(\tXc_N) {-} f(\hXc_N)|^{q(1-\delta'')}]$
and hence we have the result that 
$\EE\left[ (\Delta f)^2 \one_{A\cup B} \right]
\prec h^{(1-\delta)/2-1/q} (\EE[\, |\tZ{-}Z|^q])^{2/q}$
when $h^{1/2} \geq (\EE[\, |\tZ{-}Z|^q])^{1/(q+1)}$.

Combining the bounds from cases A and B with (\ref{eq:bound1}),
we obtain the desired final result.
\end{proof}

\section{Numerical results}

Our numerical tests are for the simplest possible example
of Geometric Brownian Motion,
\begin{equation*}
\label{eqt:gbm_sde}
\D X_t = \mu\, X_t\, \D t + \sigma\, X_t\, \D W_t.
\end{equation*}
In our simulations we take $\mu\!=\!0.05$, $\sigma\!=\!0.2$,
$T\!=\!1$, and $X_0\!=\!1$. The coarsest level $\ell\!=\!0$
uses a single time step, and higher levels use $4^\ell$
timesteps on level $\ell$ so that $h_\ell \!=\! 2^{-2l} $.

For the Normal random variables we use the 
approximations discussed in section \ref{sec:approximations}:
\begin{enumerate}\setlength{\itemsep}{-0.02in}
\item the quantised piecewise constant approximation using 1024 intervals;
\item the piecewise linear approximation on 16 dyadic intervals
on $(0,1/2)$;
\item a degree 7 polynomial approximation.
\end{enumerate}
Note that the values of $\EE[\, | \tZ{-}Z|^2 ]$ for these are
$1.5\!\times\!10^{-4}$, $4\!\times\!10^{-5}$ and $2.6\!\times\!10^{-3}$, respectively.

\begin{figure}[tb]
\centering
\includegraphics[width=0.7\linewidth]{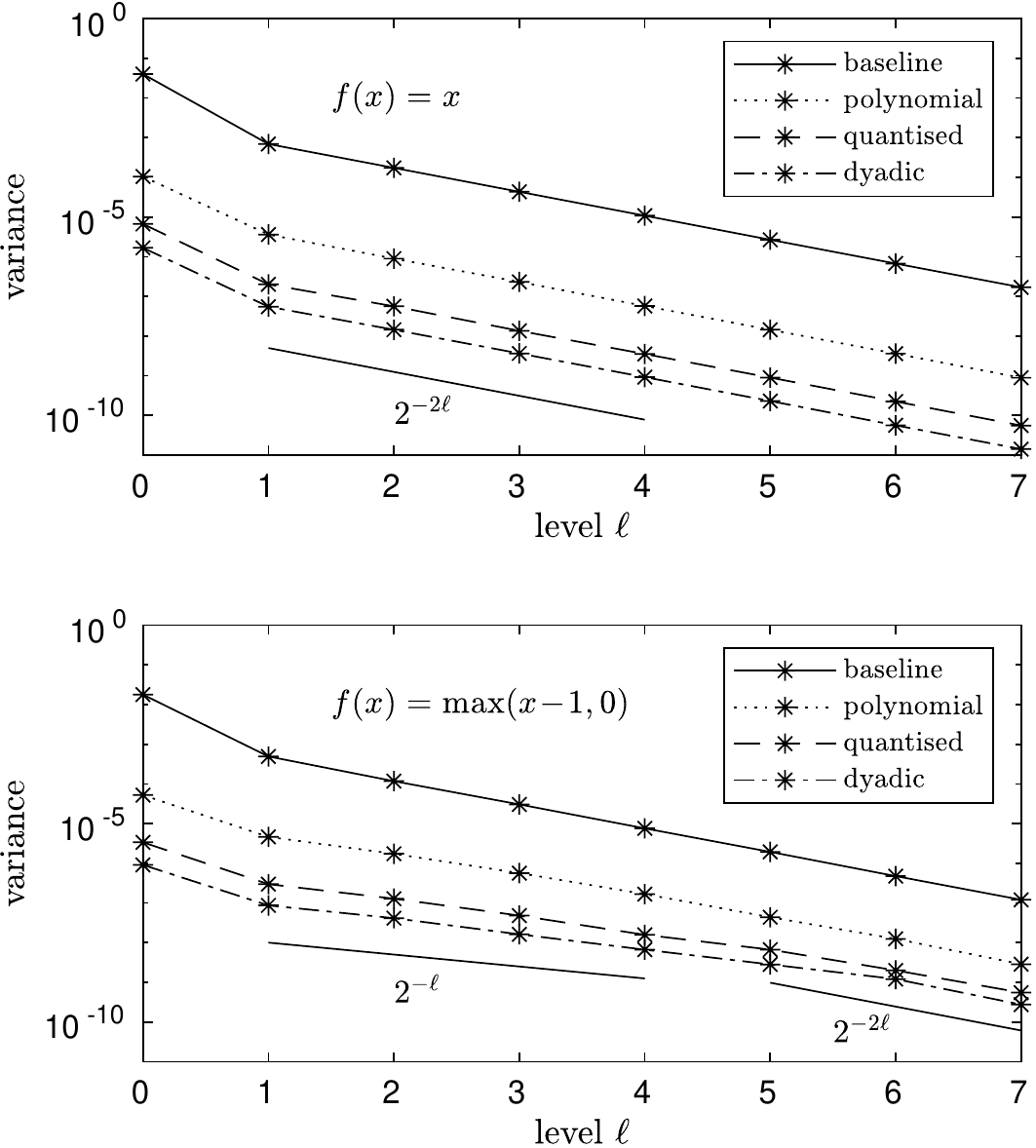}
\caption{MLMC variances for two different output functions, with
  reference lines proportional to $2^{-2\ell}$ and $2^{-\ell}$.}
\label{fig:numeric_results}
\end{figure}

Figure \ref{fig:numeric_results} presents results for all
three approximations for two different output functions,
$f(x) {\equiv} x$ and $f(x) {\equiv} \max(x{-}1,0)$.
In all cases the variances for $\VV[\tP_\ell{-}\tP_{\ell-1}]$
using the approximate Normals are visually indistinguishable
from $\VV[\hP_\ell{-}\hP_{\ell-1}]$ which is the line labelled
as ``baseline''; the other three lines are the variances
$\VV[(\hP_\ell{-}\hP_{\ell-1}) - (\tP_\ell{-}\tP_{\ell-1})]$
for the three approximations.

For the first case, $f(x) {\equiv} x$, by choosing $q$ close to 2,
Lemma \ref{lemma:polynomial} gives
\[
\tV_\ell \equiv
\VV\left[ (\hP_\ell{-} \hP_{\ell-1}) - (\tP_\ell {-} \tP_{\ell-1}) \right]
\approx O\left( 2^{-2\ell}\, \EE[\, | \tZ{-}Z|^2 ]\right).
\]
The numerical results appear to be consistent with this, with
$\tV_\ell$ decreasing with level approximately proportional to
$2^{-2\ell}$, as indicated by the reference line which is
proportional to $2^{-2\ell}$. For a fixed level $\ell$, the
variation in $\tV_\ell$ between the three different approximations
is roughly proportional to $\EE[\, | \tZ{-}Z|^2 ]$,
with the piecewise linear approximation on dyadic intervals being
the most accurate and hence giving the smallest values for $\tV_\ell$,
and the polynomial approximation being much less accurate leading
to larger values for $\tV_\ell$.

For the second case, $f(x) {\equiv} \max(x{-}1,0)$,
choosing $\delta$ close to zero, Lemma \ref{lemma:polynomial_2} gives
\[
  \tV_\ell \approx O\left(  \min\left\{ 2^{-2\ell}\, \EE[\, | \tZ{-}Z|^q ]^{1/(q+1)},
   2^{-(1-2/q)\ell}\, \EE[\, | \tZ{-}Z|^q ]^{2/q} \right\}  \right),
\]
for any $q{>}2$, whereas the earlier heuristic analysis suggested
\[
\tV_\ell \approx O\left( \min\left\{ 2^{-2\ell}\, \EE[\, |\tZ{-}Z|^2]^{1/2},
                       2^{-\ell}\, \EE[\, |\tZ{-}Z|^2] \right\} \right).
\]
The numerical results are plotted with reference lines proportional
to $2^{-\ell}$ and $2^{-2\ell}$.  The results do show a slight change in
the slope reflecting the switch from $O(2^{-\ell})$ to $O(2^{-2\ell})$
in the analysis.

Regarding the overall computational efficiency, as discussed in section
\ref{sec:nested_MLMC}, the CPU implementations using the quantised and
dyadic approximations are approximately 7 times more efficient, so
$\tC_\ell/C_\ell \approx 1/7$.  The quantity
$\sqrt{(C_\ell/\tC_\ell + 1)\, \tV_\ell/V_\ell \,}$ is approximately
0.026 and 0.052 for the output function $f(x){=}x$, using the dyadic and
quantised approximations, respectively, and
0.14 and 0.19 for the output function $f(x){=}\max(x{-}1,0)$, using the
dyadic and quantised approximations. Therefore, in all four cases
the total cost is reduced by a factor which is close to $\tC_\ell/C_\ell$.

\section{Conclusions and future work}

In this paper we have presented a general nested multilevel
Monte Carlo framework which employs approximate random variables
which can be sampled much more efficiently than the true distribution.
As a specific example, we investigated the use of approximate Normal
random variables for an Euler-Maruyama discretisation of a scalar SDE.
A detailed error analysis bounds the variance of the differences
in the SDE path approximations as a function of the error in the
approximate inverse Normal distribution. This analysis is supported
by numerical results for the simplest possible case of Geometric
Brownian Motion.

There are two directions in which we plan to extend this research.
This first is to investigate approximations of other distributions.
Two are of particular interest; one is the Poisson distribution,
which is important for continuous-time Markov processes \cite{ah12,ahs14}
and is simulated using the inverse Normal CDF \cite{giles16},
and the other is the non-central $\chi^2$-distribution which is
important for simulating the  \emph{Cox-Ingersoll-Ross} (CIR) process
which is used extensively in computational finance.  In both cases, the 
computational savings may be greater, but it may prove to be very 
difficult to carry out a detailed numerical analysis of the resulting 
MLMC variances.

The second direction is to use reduced precision computer arithmetic
in performing the calculations of $\tXn$, further reducing the cost
of the approximate calculations.  This builds on prior research by
others, implementing MLMC methods on FPGAs (field-programmable gate
arrays) \cite{bswohrkk14,ohrbswk15}.
The rounding error effect of finite precision arithmetic can be
modelled as an additional random error at each timestep; it is expected
that on the coarsest levels with few timesteps this additional error
will be small, but on the finest levels it may become significant and
so perhaps such levels should be computed using single precision.

\section*{Acknowledgements}

This publication is based on work supported by both the ICONIC
EPSRC Programme Grant (EP/P020720/1)
and the EPSRC Centre for Doctoral Training in Industrially Focused
Mathematical Modelling (EP/L015803/1) in collaboration with Arm Ltd.

%\newpage

\bibliographystyle{unsrt}
\bibliography{mlmc,mc}

\appendix

\newpage
\section{Mean Value Theorem and a generalisation}

\begin{lemma}[Mean Value Theorem]
\label{lemma:MVT}
If $f: \RR \rightarrow \RR$ is $C^1(\RR)$, then there exists 
$\xi$ which is a positively-weighted average of $x_1, x_2$
(i.e., $\xi = s\, x_1 + (1{-}s) x_2$ for some $0\!<\!s\!<\!1$)
such that
\[
f(x_1) - f(x_2) = (x_1{-}x_2) \, f'(\xi).
\]
\end{lemma}

\begin{lemma}
\label{lemma:MVT2}
If $f: \RR \rightarrow \RR$ is $C^1(\RR)$, and $f'$ is Lipschitz 
continuous with Lipschitz constant $L'_f$ then there exists $\xi$ 
which is a positively-weighted average of $x_1, x_2, x_3, x_4$ such
that
\[
f(x_1) - f(x_2) - f(x_3) + f(x_4) = 
(x_1 - x_2 - x_3 + x_4) f'(\xi) + R,
\]
where
\[
|R| \leq \fracs{1}{2} L'_f (|x_1{-}x_2| + |x_3{-}x_4|)\ (|x_1{-}x_3| + |x_2{-}x_4|).
\]
\end{lemma}
\begin{proof}
Without loss of generality, we can assume 
\begin{equation}
|x_1{-}x_2| + |x_3{-}x_4|\ \leq\ |x_1{-}x_3| + |x_2{-}x_4|,
\label{eq:assmp}
\end{equation}
since otherwise we can just swap $x_2$ and $x_3$.

Now, using Lemma \ref{lemma:MVT} we get
\begin{eqnarray*}
f(x_1) - f(x_2) &{=}& (x_1{-}x_2) \, f'(\xi_1),\\
f(x_3) - f(x_4) &{=}& (x_3{-}x_4) \, f'(\xi_2),
\end{eqnarray*}
where $\xi_1$ and $\xi_2$ are positively-weighted averages 
of $x_1, x_2$ and $x_3, x_4$, respectively.
Taking the difference gives
\[
f(x_1) - f(x_2) - f(x_3) + f(x_4) = 
  \fracs{1}{2} (x_1 {-} x_2 {-} x_3 {+} x_4)  (f'(\xi_1) {+} f'(\xi_2)) + R,
\]
where
\[
R = \fracs{1}{2} (x_1 {-} x_2 {+} x_3 {-} x_4)  (f'(\xi_1) {-} f'(\xi_2)).
\]
Since $f'$ is continuous, there exists an $\xi$ which is a positively-weighted
of $\xi_1$ and $\xi_2$, and hence of $x_1, x_2, x_3, x_4$ such that
\[
\fracs{1}{2}  (f'(\xi_1) + f'(\xi_2)) = f'(\xi).
\]
Note that
\[
\xi_1 - \xi_2 = (\xi_1 - \fracs{1}{2}(x_1+x_2)) + (\xi_2 - \fracs{1}{2}(x_3+x_4))
+ ( \fracs{1}{2}(x_1+x_2) - \fracs{1}{2}(x_3+x_4) )
\]
and therefore, due to (\ref{eq:assmp}),
\begin{eqnarray*}
| \xi_1 - \xi_2 | 
&\leq& \fracs{1}{2}|x_1-x_2| + \fracs{1}{2}|x_3-x_4| + \fracs{1}{2}|x_1-x_3| + \fracs{1}{2}|x_2-x_4|\\
&\leq& |x_1-x_3| + |x_2-x_4|
\end{eqnarray*}
Hence, due to the Lipschitz property of $f'$, 
\[
|R| \leq \fracs{1}{2} L'_f (|x_1{-}x_2| + |x_3{-}x_4|)\ (|x_1{-}x_3| + |x_2{-}x_4|).
\]
\end{proof}

\end{document}